\documentclass{article}

\usepackage{latexsym}
\usepackage{amsthm}

\usepackage{fullpage}

\usepackage{hyperref}
\usepackage{graphicx}
\usepackage{pgf,tikz}
\usepackage{amssymb,latexsym}
\usepackage{amsmath}
\usepackage{url}
\usepackage{nicefrac}

\usepackage[footnotesize]{caption}

\usetikzlibrary{trees,decorations}
\usetikzlibrary{arrows,automata}

\newtheorem{theorem}{Theorem}
\newtheorem{lemma}[theorem]{Lemma}

\newcommand{\beq}{\begin{equation}}
\newcommand{\eeq}{\end{equation}}
\newcommand{\bea}{\begin{array}}
\newcommand{\eea}{\end{array}}

\newcommand{\C}{\mathcal{C}}

\newcommand{\F}{\mathcal{F}}

\renewcommand{\S}{\mathcal{S}}
\renewcommand{\C}{\mathcal{C}}
\renewcommand{\L}{\mathcal{D}}
\renewcommand{\F}{\mathcal{F}}
\newcommand{\Ls}{\L^\star}
\newcommand{\Fs}{\F^\star}

\begin{document}

\title{\bf On Enumeration of Dyck--Schr\"oder Paths}

\author{\Large Max A. Alekseyev\\The George Washington University\\Washington, DC, U.S.A.\\Email: \texttt{maxal@gwu.edu}
}


\date{}

\maketitle

\begin{abstract}
We address the problem of enumerating paths in square lattices, where allowed steps include $(1,0)$ and $(0,1)$ everywhere, and $(1,1)$ above the diagonal $y=x$.
We consider two such lattices differing in whether the $(1,1)$ steps are allowed along the diagonal itself.
Our analysis leads to explicit generating functions and an efficient way to compute terms of many sequences in the Online Encyclopedia of Integer Sequences, 
proposed by Clark Kimberling almost two decades ago.
\end{abstract}

\section{Introduction}

The \emph{Catalan numbers} $C_n=\frac{1}{n+1}\binom{2n}{n}$ (sequence \texttt{A000108} in the OEIS~\cite{OEIS}) 
enumerate among other combinatorial objects~\cite{Stanley2015} the paths from $(0,0)$ to $(n,n)$ 
in the integer lattice bounded by lines $y=0$ and $y=x$ with unit steps $(0,1)$ and $(1,0)$, called \emph{Dyck paths}. 
We represent these restrictions as a directed diagram, which we will refer to as the \emph{Dyck diagram} $\L_D$~(Fig.~\ref{fig:CSlat}, left panel).
If we allow diagonal steps $(1,1)$ in this lattice, the paths in it become known as \emph{Schr\"oder paths}~\cite{Petersen2015}, 
and the number of such paths from $(0,0)$ to $(n,n)$ is given by the \emph{large Schr\"oder numbers} $S_n$ (sequence \texttt{A006318} in the OEIS). 
We similarly represent these restrictions as a directed diagram, which refer to as the \emph{Schr\"oder diagram} $\L_S$~(Fig.~\ref{fig:CSlat}, right panel).

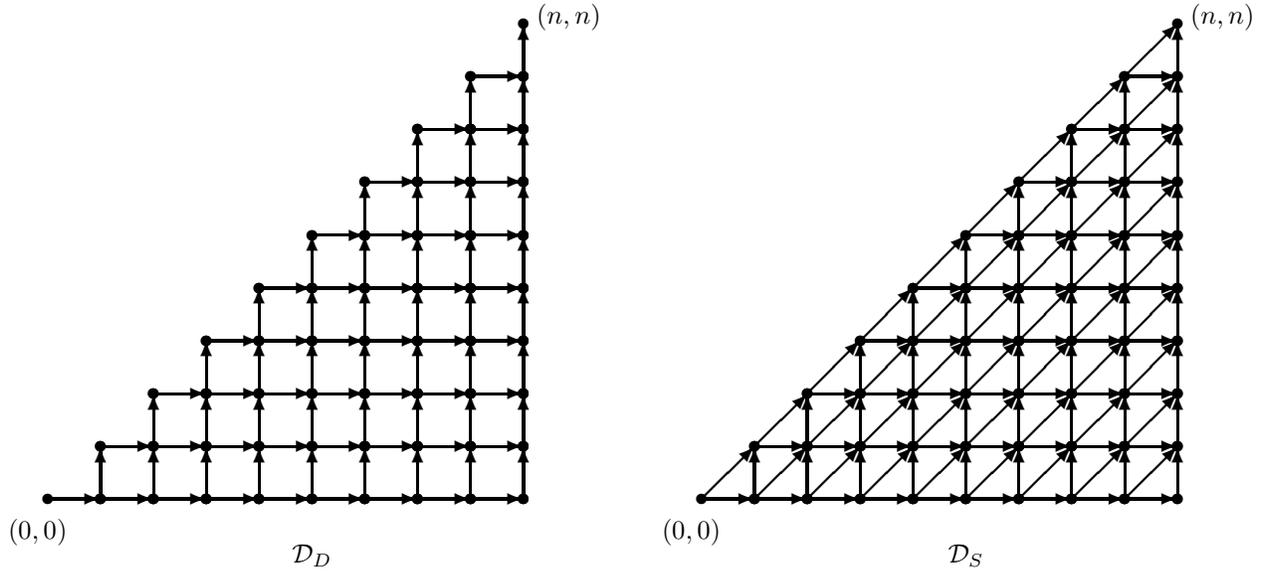
\begin{figure}[!t]
\begin{center}
\begin{tabular}{ccc}
\begin{picture}(200,200)
\put(-15,-15){$(0,0)$}
\put(185,180){$(n,n)$}
\thicklines
\multiput(20,0)(20,0){9}{\put(0,0){\circle*{4}}\put(0,0){\vector(0,1){20}}\put(0,20){\circle*{4}}}
\multiput(40,20)(20,0){8}{\put(0,0){\circle*{4}}\put(0,0){\vector(0,1){20}}\put(0,20){\circle*{4}}}
\multiput(60,40)(20,0){7}{\put(0,0){\circle*{4}}\put(0,0){\vector(0,1){20}}\put(0,20){\circle*{4}}}
\multiput(80,60)(20,0){6}{\put(0,0){\circle*{4}}\put(0,0){\vector(0,1){20}}\put(0,20){\circle*{4}}}
\multiput(100,80)(20,0){5}{\put(0,0){\circle*{4}}\put(0,0){\vector(0,1){20}}\put(0,20){\circle*{4}}}
\multiput(120,100)(20,0){4}{\put(0,0){\circle*{4}}\put(0,0){\vector(0,1){20}}\put(0,20){\circle*{4}}}
\multiput(140,120)(20,0){3}{\put(0,0){\circle*{4}}\put(0,0){\vector(0,1){20}}\put(0,20){\circle*{4}}}
\multiput(160,140)(20,0){2}{\put(0,0){\circle*{4}}\put(0,0){\vector(0,1){20}}\put(0,20){\circle*{4}}}
\multiput(180,160)(20,0){1}{\put(0,0){\circle*{4}}\put(0,0){\vector(0,1){20}}\put(0,20){\circle*{4}}}

\multiput(0,0)(0,20){1}{\put(0,0){\circle*{4}}\put(0,0){\vector(1,0){20}}\put(20,0){\circle*{4}}}
\multiput(20,0)(0,20){2}{\put(0,0){\circle*{4}}\put(0,0){\vector(1,0){20}}\put(20,0){\circle*{4}}}
\multiput(40,0)(0,20){3}{\put(0,0){\circle*{4}}\put(0,0){\vector(1,0){20}}\put(20,0){\circle*{4}}}
\multiput(60,0)(0,20){4}{\put(0,0){\circle*{4}}\put(0,0){\vector(1,0){20}}\put(20,0){\circle*{4}}}
\multiput(80,0)(0,20){5}{\put(0,0){\circle*{4}}\put(0,0){\vector(1,0){20}}\put(20,0){\circle*{4}}}
\multiput(100,0)(0,20){6}{\put(0,0){\circle*{4}}\put(0,0){\vector(1,0){20}}\put(20,0){\circle*{4}}}
\multiput(120,0)(0,20){7}{\put(0,0){\circle*{4}}\put(0,0){\vector(1,0){20}}\put(20,0){\circle*{4}}}
\multiput(140,0)(0,20){8}{\put(0,0){\circle*{4}}\put(0,0){\vector(1,0){20}}\put(20,0){\circle*{4}}}
\multiput(160,0)(0,20){9}{\put(0,0){\circle*{4}}\put(0,0){\vector(1,0){20}}\put(20,0){\circle*{4}}}
\end{picture}
&~\qquad~&
\begin{picture}(200,200)
\put(-15,-15){$(0,0)$}
\put(185,180){$(n,n)$}
\thicklines
\multiput(20,0)(20,0){9}{\put(0,0){\circle*{4}}\put(0,0){\vector(0,1){20}}\put(0,20){\circle*{4}}}
\multiput(40,20)(20,0){8}{\put(0,0){\circle*{4}}\put(0,0){\vector(0,1){20}}\put(0,20){\circle*{4}}}
\multiput(60,40)(20,0){7}{\put(0,0){\circle*{4}}\put(0,0){\vector(0,1){20}}\put(0,20){\circle*{4}}}
\multiput(80,60)(20,0){6}{\put(0,0){\circle*{4}}\put(0,0){\vector(0,1){20}}\put(0,20){\circle*{4}}}
\multiput(100,80)(20,0){5}{\put(0,0){\circle*{4}}\put(0,0){\vector(0,1){20}}\put(0,20){\circle*{4}}}
\multiput(120,100)(20,0){4}{\put(0,0){\circle*{4}}\put(0,0){\vector(0,1){20}}\put(0,20){\circle*{4}}}
\multiput(140,120)(20,0){3}{\put(0,0){\circle*{4}}\put(0,0){\vector(0,1){20}}\put(0,20){\circle*{4}}}
\multiput(160,140)(20,0){2}{\put(0,0){\circle*{4}}\put(0,0){\vector(0,1){20}}\put(0,20){\circle*{4}}}
\multiput(180,160)(20,0){1}{\put(0,0){\circle*{4}}\put(0,0){\vector(0,1){20}}\put(0,20){\circle*{4}}}

\multiput(0,0)(0,20){1}{\put(0,0){\circle*{4}}\put(0,0){\vector(1,0){20}}\put(20,0){\circle*{4}}}
\multiput(20,0)(0,20){2}{\put(0,0){\circle*{4}}\put(0,0){\vector(1,0){20}}\put(20,0){\circle*{4}}}
\multiput(40,0)(0,20){3}{\put(0,0){\circle*{4}}\put(0,0){\vector(1,0){20}}\put(20,0){\circle*{4}}}
\multiput(60,0)(0,20){4}{\put(0,0){\circle*{4}}\put(0,0){\vector(1,0){20}}\put(20,0){\circle*{4}}}
\multiput(80,0)(0,20){5}{\put(0,0){\circle*{4}}\put(0,0){\vector(1,0){20}}\put(20,0){\circle*{4}}}
\multiput(100,0)(0,20){6}{\put(0,0){\circle*{4}}\put(0,0){\vector(1,0){20}}\put(20,0){\circle*{4}}}
\multiput(120,0)(0,20){7}{\put(0,0){\circle*{4}}\put(0,0){\vector(1,0){20}}\put(20,0){\circle*{4}}}
\multiput(140,0)(0,20){8}{\put(0,0){\circle*{4}}\put(0,0){\vector(1,0){20}}\put(20,0){\circle*{4}}}
\multiput(160,0)(0,20){9}{\put(0,0){\circle*{4}}\put(0,0){\vector(1,0){20}}\put(20,0){\circle*{4}}}

\multiput(0,0)(20,20){9}{\vector(1,1){20}}
\multiput(20,0)(20,20){8}{\vector(1,1){20}}
\multiput(40,0)(20,20){7}{\vector(1,1){20}}
\multiput(60,0)(20,20){6}{\vector(1,1){20}}
\multiput(80,0)(20,20){5}{\vector(1,1){20}}
\multiput(100,0)(20,20){4}{\vector(1,1){20}}
\multiput(120,0)(20,20){3}{\vector(1,1){20}}
\multiput(140,0)(20,20){2}{\vector(1,1){20}}
\multiput(160,0)(20,20){1}{\vector(1,1){20}}
\end{picture}\\
\\
$\L_D$ & & $\L_S$
\end{tabular}
\end{center}
\caption{Dyck diagram $\L_D$ and Schr\"oder diagram $\L_S$.}
\label{fig:CSlat}
\end{figure}

In 1999, Clark Kimberling contributed the sequences \texttt{A026769}--\texttt{A026790} to the OEIS~\cite{OEIS}, concerning a composition of the Dyck and Schr\"oder diagrams
that below the diagonal $y=x$ represents the Dyck diagram and above the diagonal represents the transposed Schr\"oder diagram.
There are two such diagrams $\L_{DS}$ and $\Ls_{DS}$, where the $(1,1)$ steps along the diagonal are allowed and disallowed, respectively (Fig.~\ref{fig:CSL}).
In the present work, we address the problem of enumerating \emph{Dyck--Schr\"oder paths}, i.e., paths in the diagrams $\L_{DS}$ and $\Ls_{DS}$.
We begin our analysis by recalling some useful facts about the Dyck and Schr\"oder diagrams.

The generating functions for Catalan and Schr\"oder numbers are given by
$$\C(x) = \sum_{n=0}^{\infty} C_n\cdot x^n = \frac{1-\sqrt{1-4x}}{2x}$$
and
$$\S(x) = \sum_{n=0}^{\infty} S_n\cdot x^n = \frac{1-x-\sqrt{1-6x+x^2}}{2x} = \frac{1}{1-x}\cdot \C\left(\frac{x}{(1-x)^2}\right),$$
respectively~\cite{Petersen2015}. 
So the number of paths from $(0,0)$ to $(n,n)$ in $\L_D$ and $\L_S$ is given by $[x^n]\ \C(x)$ and $[x^n]\ \S(x)$, respectively, 
where $[x^n]$ denotes the operator of taking the coefficient of $x^n$.

We will need the following lemma, which states well-known facts about the number of paths and the number of \emph{subdiagonal paths} 
(i.e., paths that lay below the line $y=x$, except possibly for their endpoints) from $(0,0)$ to $(n,k)$ in $\L_D$ and $\L_S$.

\begin{lemma}\label{lem:paths}
For any integers $n\geq k\geq 0$, 
\begin{itemize}
\item[(i)]
the number of paths from $(0,0)$ to $(n,k)$ in $\L_D$ and $\L_S$ equals $[x^k]\ \C(x)^{n-k+1}$ and $[x^k]\ \S(x)^{n-k+1}$, respectively.
\item[(ii)]
the number of subdiagonal paths from $(0,0)$ to $(n,n)$ in $\L_D$ and $\L_S$ equals $[x^{n-1}]\ \C(x)$ and $[x^{n-1}]\ \S(x)$, respectively.
\item[(iii)]
for $n>k$, the number of subdiagonal paths from $(0,0)$ to $(n,k)$ in $\L_D$ and $\L_S$ equals $[x^k]\ \C(x)^{n-k}$ and $[x^k]\ \S(x)^{n-k}$, respectively.
\end{itemize}
\end{lemma}

\begin{figure}[!t]
\begin{center}
\begin{tabular}{ccc}
\begin{picture}(200,200)
\put(-15,-15){$(0,0)$}
\put(185,180){$(n,n)$}
\thicklines
\multiput(0,0)(20,0){10}{\put(0,0){\circle*{4}}\multiput(0,0)(0,20){9}{\vector(0,1){20}\put(0,0){\circle*{4}}}}
\multiput(0,0)(0,20){10}{\put(0,0){\circle*{4}}\multiput(0,0)(20,0){9}{\vector(1,0){20}\put(0,0){\circle*{4}}}}

\multiput(0,0)(20,20){9}{\vector(1,1){20}}
\multiput(0,20)(20,20){8}{\vector(1,1){20}}
\multiput(0,40)(20,20){7}{\vector(1,1){20}}
\multiput(0,60)(20,20){6}{\vector(1,1){20}}
\multiput(0,80)(20,20){5}{\vector(1,1){20}}
\multiput(0,100)(20,20){4}{\vector(1,1){20}}
\multiput(0,120)(20,20){3}{\vector(1,1){20}}
\multiput(0,140)(20,20){2}{\vector(1,1){20}}
\multiput(0,160)(20,20){1}{\vector(1,1){20}}
\end{picture}
&~\qquad~&
\begin{picture}(200,200)
\put(-15,-15){$(0,0)$}
\put(185,180){$(n,n)$}
\thicklines
\multiput(0,0)(20,0){10}{\put(0,0){\circle*{4}}\multiput(0,0)(0,20){9}{\vector(0,1){20}\put(0,0){\circle*{4}}}}
\multiput(0,0)(0,20){10}{\put(0,0){\circle*{4}}\multiput(0,0)(20,0){9}{\vector(1,0){20}\put(0,0){\circle*{4}}}}

\multiput(0,20)(20,20){8}{\vector(1,1){20}}
\multiput(0,40)(20,20){7}{\vector(1,1){20}}
\multiput(0,60)(20,20){6}{\vector(1,1){20}}
\multiput(0,80)(20,20){5}{\vector(1,1){20}}
\multiput(0,100)(20,20){4}{\vector(1,1){20}}
\multiput(0,120)(20,20){3}{\vector(1,1){20}}
\multiput(0,140)(20,20){2}{\vector(1,1){20}}
\multiput(0,160)(20,20){1}{\vector(1,1){20}}
\end{picture}\\
\\
$\L_{DS}$ & & $\Ls_{DS}$
\end{tabular}
\end{center}
\caption{Dyck--Schr\"oder diagrams $\L_{DS}$ and $\Ls_{DS}$.}
\label{fig:CSL}
\end{figure}

\section{Enumeration of Dyck--Schr\"oder paths}

\begin{theorem}\label{th:FFs} 
Let $f_n$ and $f_n^\star$ be the number of paths from $(0,0)$ to $(n,n)$ in the diagrams $\L_{DS}$ and $\Ls_{DS}$, respectively. Then
$$\Fs(x) = \sum_{n=0}^{\infty} f_n^\star\cdot x^n = \frac{1}{1-x\cdot(\C(x)+\S(x))}$$
and
$$\F(x) = \sum_{n=0}^{\infty} f_n\cdot x^n = \frac{1}{1-x\cdot(\C(x)+\S(x)+1)} = \frac{\S(x)}{1-x\cdot\C(x)\cdot\S(x)}.$$
\end{theorem}

\begin{proof} Any path from $(0,0)$ to $(n,n)$ in $\Ls_{DS}$ consists of subdiagonal and/or supdiagonal\footnote{Similarly to subdiagonal paths,
we define supdiagonal paths as those that lay above the diagonal $y=x$, except possibly for their endpoints.}
subpaths from $(p_0,p_0)$ to $(p_1,p_1)$, from $(p_1,p_1)$ to $(p_2,p_2)$, $\ldots$, from $(p_{m-1},p_{m-1})$ to $(p_m,p_m)$, where $m\geq 0$ and
$0=p_0<p_1<\dots<p_m=n$ are integers (in other words, $p_i$ represent the coordinates of vertices where the path visits the diagonal). 
By Lemma~\ref{lem:paths}, for any $k=0,1,\dots,m-1$, the number of subdiagonal and/or supdiagonal paths from $(p_k,p_k)$ to $(p_{k+1},p_{k+1})$ equals
$$[x^{p_{k+1}-p_{k}-1}]\ (\C(x)+\S(x)).$$
Indeed, every such path is either subdiagonal (enumerated by $\C(x)$) or supdiagonal (enumerated by $\S(x)$). 

Hence, the total number of paths in from $(0,0)$ to $(n,n)$ in $\Ls_{DS}$ equals
\[
\begin{split}
& \sum_{m=0}^{\infty} \sum_{0<p_1<\dots<p_{m-1}<n} \prod_{k=0}^{m-1} [x^{p_{k+1}-p_{k}-1}]\ (\C(x)+\S(x)) \\
= & \sum_{m=0}^{\infty} [x^{n-m}]\ (\C(x)+\S(x))^m = [x^n]\ \sum_{m=0}^\infty \left(x\cdot (\C(x)+\S(x))\right)^m \\
= & [x^n]\ \frac{1}{1-x\cdot (\C(x)+\S(x))}.
\end{split}
\]

In the diagram $\L_{DS}$, in addition to subdiagonal and/or supdiagonal subpaths, we need to account for single diagonal steps (when $p_{k+1}=p_k+1$), 
which brings the additional summand 1 to $\C(x)+\S(x)$. That is, the total number of paths in from $(0,0)$ to $(n,n)$ in $\L_{DS}$ equals the coefficient of $x^n$ in
$$\frac{1}{1-x\cdot (\C(x)+\S(x)+1)} = \frac{\S(x)}{1-x\cdot\C(x)\cdot\S(x)}.$$
The last equality follows from the algebraic identity:
$$x\cdot \S(x)^2 - (1-x)\cdot \S(x) + 1 = 0.$$
\end{proof}

\begin{theorem}\label{th:CSpaths}
For integers $n,k$, the number of paths from $(0,0)$ to $(n,k)$ in $\L_{DS}$ equals 
$$
\begin{cases}
[x^k]\ \F(x)\cdot\C(x)^{n-k}, & \text{if}\ n\geq k;\\
[x^n]\ \F(x)\cdot\S(x)^{k-n}, & \text{if}\ n\leq k.
\end{cases}
$$
Similarly, the number of paths from $(0,0)$ to $(n,k)$ in $\Ls_{DS}$ equals 
$$
\begin{cases}
[x^k]\ \Fs(x)\cdot\C(x)^{n-k}, & \text{if}\ n\geq k;\\
[x^n]\ \Fs(x)\cdot\S(x)^{k-n}, & \text{if}\ n\leq k.
\end{cases}
$$
\end{theorem}

\begin{proof}
Any path from $(0,0)$ to $(n,k)$ in $\L_{DS}$ is formed by a path from $(0,0)$ to $(m,m)$ for some $m\geq 0$ and a path from $(m,m)$ to $(n,k)$ that never visits the diagonal again.
Clearly, this decomposition is unique and so is $m$.
The number of paths from $(0,0)$ to $(m,m)$ equals $[x^m]\ \F(x)$.
If $n>k$, then the number of paths from $(m,m)$ to $(n,k)$ avoiding the diagonal equals 
the number of subdiagonal paths from $(0,0)$ to $(n-m,k-m)$ in $\L_D$, which is $[x^{k-m}]\ \C(x)^{n-k}$ (by Lemma~\ref{lem:paths}).
In this case, the number of paths $(0,0)$ to $(n,k)$ in $\L_{DS}$ equals
$$\sum_{m=0}^{\infty} [x^m]\ \F(x)\cdot [x^{n-m}]\ \C(x)^{n-k} = [x^n]\ \F(x)\cdot\C(x)^{n-k}.$$
Similarly, if $n<k$, then the number of paths from $(m,m)$ to $(n,k)$ avoiding the diagonal equals 
the number of subdiagonal paths from $(0,0)$ to $(k-m,n-m)$ in $\L_S$, which is $[x^{n-m}]\ \S(x)^{k-n}$ (by Lemma~\ref{lem:paths}). In this case, 
the number of paths $(0,0)$ to $(n,k)$ in $\L_{DS}$ equals $[x^n]\ \F(x)\cdot\S(x)^{k-n}$.
It is easy to see that the formulae in both cases are also consistent with the case $n=k$, where the number of paths equals $[x^n]\ \F(x)$ by the definition of $\F(x)$.

The diagram $\Ls_{DS}$ is considered similarly.
\end{proof}

\section{Sequences in the OEIS}

Below we derive formulae for sequences \texttt{A026769}--\texttt{A026779} (concerning the diagram $\Ls_{DS}$) and 
sequences \texttt{A026780}--\texttt{A026790} (concerning the diagram $\L_{DS}$) in the Online Encyclopedia of Integer Sequences~\cite{OEIS}.

\subsection{Sequences \texttt{A026769} and \texttt{A026780}}
$\texttt{A026769}(n,k)$ and $\texttt{A026780}(n,k)$ give the number of paths from $(0,0)$ to $(k,n-k)$ in the diagrams $\Ls_{DS}$ and $\L_{DS}$, respectively.
Formulae for these numbers are given in Theorem~\ref{th:CSpaths}.

\subsection{Sequences \texttt{A026770}--\texttt{A026774} and \texttt{A026781}--\texttt{A026785}}
The $n$-th term of \texttt{A026770}--\texttt{A026774} enumerate paths in $\Ls_{DS}$ from $(0,0)$ to $(n,n)$, $(n-1,n+1)$, $(n-2,n+2)$, $(n-1,n)$, and $(n-2,n+1)$, respectively.
By Theorem~\ref{th:CSpaths}, the ordinary generating function for the number of such paths is $\Fs(x)$, $x\cdot \Fs(x)\cdot\S(x)^2$, $x^2\cdot \Fs(x)\cdot\S(x)^4$, $x\cdot \Fs(x)\cdot\S(x)$, and $x^2\cdot \Fs(x)\cdot\S(x)^3$, respectively.

The sequences \texttt{A026781}--\texttt{A026785} enumerate similar paths in $\L_{DS}$ and have ordinary generating functions
$\F(x)$, $x\cdot \F(x)\cdot\S(x)^2$, $x^2\cdot \F(x)\cdot\S(x)^4$, $x\cdot \F(x)\cdot\S(x)$, and $x^2\cdot \F(x)\cdot\S(x)^3$, respectively.

\subsection{Sequences \texttt{A026775} and \texttt{A026786}}
The terms $\texttt{A026775}(n)$ and $\texttt{A026786}(n)$ give the number of paths from $(0,0)$ to $(\lfloor \nicefrac{n}{2}\rfloor, \lceil \nicefrac{n}{2}\rceil)$ in the diagrams $\Ls_{DS}$ and $\L_{DS}$, respectively.
By Theorem~\ref{th:CSpaths}, these are the coefficients of $x^{\lfloor \nicefrac{n}{2}\rfloor}$ in $\Fs(x)\cdot\S(x)^{n\bmod 2}$ and $\F(x)\cdot\S(x)^{n\bmod 2}$, 
which are the same as the coefficients of $x^n$ in $\Fs(x^2)\cdot(1+x\cdot \S(x^2))$ and $\F(x)\cdot(1+x\cdot \S(x^2))$.

\subsection{Sequences \texttt{A026776} and \texttt{A026787}}
The terms $\texttt{A026776}(n)$ and $\texttt{A026787}(n)$ give the total number of paths from $(0,0)$ to $(i,n-i)$, where $i=0,1,\dots,n$, in the diagrams $\Ls_{DS}$ and $\L_{DS}$, respectively.

\begin{theorem}\label{th:A026776}
The ordinary generating function for \texttt{A026776} is
$$\Fs(x^2)\cdot \left( \frac{1}{1-x\cdot \S(x^2)} + \frac{1}{1-x\cdot \C(x^2)} - 1\right).$$
The ordinary generating function for \texttt{A026787} is
$$\F(x^2)\cdot \left( \frac{1}{1-x\cdot \S(x^2)} + \frac{1}{1-x\cdot \C(x^2)} - 1\right).$$
\end{theorem}

\begin{proof}
By Theorem~\ref{th:CSpaths}, $\texttt{A026776}(n)$ equals
\[
\begin{split}
& \sum_{i=0}^{\lfloor \nicefrac{n}{2}\rfloor} [x^i]\ \Fs(x)\cdot\S(x)^{n-2i} + \sum_{i=\lfloor \nicefrac{n}{2}\rfloor+1}^n [x^{n-i}]\ \Fs(x)\cdot\C(x)^{2i-n} \\
= & \sum_{i=0}^{\lfloor \nicefrac{n}{2}\rfloor} [x^{2i}]\ \Fs(x^2)\cdot\S(x^2)^{n-2i} + \sum_{i=\lfloor \nicefrac{n}{2}\rfloor+1}^n [x^{2n-2i}]\ \Fs(x^2)\cdot\C(x^2)^{2i-n} \\
= & [x^n]\ \Fs(x^2)\cdot\left(\sum_{i=0}^{\lfloor \nicefrac{n}{2}\rfloor} \left(x\cdot \S(x^2)\right)^{n-2i} + \sum_{i=\lfloor \nicefrac{n}{2}\rfloor+1}^n \left(x\cdot \C(x^2)\right)^{2i-n}\right).
\end{split}
\]
We notice that in the first sum, the powers of $x\cdot \S(x^2)$ go over the nonnegative integers up to $n$ of the same oddness as $n$. Since we are interested only in the coefficient of $x^n$, 
we can drop both these restrictions. Namely, the powers above $n$ have the coefficient of $x^n$ equal zero, while the power $m$ of the opposite oddness than $n$ 
may have nonzero coefficients only for $x$ in powers of the same oddness as $m$. The same arguments apply for the powers of $x\cdot \C(x^2)$, except that in this case they go over the positive integers.
That is, the above expression simplifies to
\[
\begin{split}
& [x^n]\ \Fs(x^2)\cdot\left(\sum_{m=0}^\infty \left(x\cdot \S(x^2)\right)^m + \sum_{m=1}^\infty \left(x\cdot \C(x^2)\right)^m\right) \\
= & [x^n]\ \Fs(x^2)\cdot \left( \frac{1}{1-x\cdot \S(x^2)} + \frac{1}{1-x\cdot \C(x^2)} - 1\right),
\end{split}
\]
which gives the ordinary generating function for \texttt{A026776}.

The generating function for \texttt{A026787} is be obtained by replacing $\Fs(x^2)$ with $\F(x^2)$.
\end{proof}

\subsection{Sequences \texttt{A026777} and \texttt{A026788}}
The term $\texttt{A026777}(n)$ gives the total number of paths in $\Ls_{DS}$ from $(0,0)$ to $(i,n-i)$, where $i=0,1,\dots,\lfloor \nicefrac{n}{2}\rfloor$.
The ordinary generating function for \texttt{A026777} is $\frac{\Fs(x^2)}{1-x\cdot \S(x^2)}$,
which can be easily obtained from our analysis of the sequence \texttt{A026776} above.

The ordinary generating function for \texttt{A026788}, which enumerates similar paths in $\L_{DS}$, equals $\frac{\F(x^2)}{1-x\cdot \S(x^2)}$.

\subsection{Sequences \texttt{A026778} and \texttt{A026789}}
The term $\texttt{A026778}(n)$ gives the total number of paths in $\Ls_{DS}$ from $(0,0)$ to $(i,i-j)$, where $0\leq j\leq i\leq n$.
It is easy to see that
$$\texttt{A026778}(n) = \sum_{m=0}^n \texttt{A026776}(m)$$
and thus the ordinary generating function for \texttt{A026778} can be obtained from the one for \texttt{A026776} by multiplying it by $\frac{1}{1-x}$. That is, 
the ordinary generating function for \texttt{A026778} equals
$$\frac{\Fs(x^2)}{1-x}\cdot \left( \frac{1}{1-x\cdot \S(x^2)} + \frac{1}{1-x\cdot \C(x^2)} - 1\right).$$

The ordinary generating function for \texttt{A026789}, which enumerates similar paths in $\L_{DS}$, equals
$$\frac{\F(x^2)}{1-x}\cdot \left( \frac{1}{1-x\cdot \S(x^2)} + \frac{1}{1-x\cdot \C(x^2)} - 1\right).$$

\subsection{Sequences \texttt{A026779} and \texttt{A026790}}
The terms $\texttt{A026779}(n)$ and $\texttt{A026790}(n)$ give the total number of paths from $(0,0)$ to $(i,n-2i)$, where $i=0,1,\dots,\lfloor \nicefrac{n}{2}\rfloor$, in the diagrams 
$\Ls_{DS}$ and $\L_{DS}$, respectively.

\begin{theorem}\label{th:A026779}
The ordinary generating function for \texttt{A026779} is
$$\Fs(x^3)\cdot \left( \frac{1}{1-x\cdot \S(x^3)} + \frac{1}{1-x^2\cdot \C(x^3)} - 1\right).$$
The ordinary generating function for \texttt{A026790} is
$$\F(x^3)\cdot \left( \frac{1}{1-x\cdot \S(x^3)} + \frac{1}{1-x^2\cdot \C(x^3)} - 1\right).$$
\end{theorem}

\begin{proof} By Theorem~\ref{th:CSpaths}, $\texttt{A026779}(n)$ equals
\[
\begin{split}
& \sum_{i=0}^{\lfloor \nicefrac{n}{3}\rfloor} [x^i]\ \Fs(x)\cdot\S(x)^{n-3i} + \sum_{i=\lfloor \nicefrac{n}{3}\rfloor+1}^{\lfloor \nicefrac{n}{2}\rfloor} [x^{n-2i}]\ \Fs(x)\cdot\C(x)^{3i-n} \\
= & \sum_{i=0}^{\lfloor \nicefrac{n}{3}\rfloor} [x^{3i}]\ \Fs(x^3)\cdot\S(x^3)^{n-3i} + \sum_{i=\lfloor \nicefrac{n}{3}\rfloor+1}^{\lfloor \nicefrac{n}{2}\rfloor} [x^{3n-6i}]\ \Fs(x^3)\cdot\C(x^3)^{3i-n} \\
= & [x^n]\ \Fs(x^3)\cdot\left(\sum_{i=0}^{\lfloor \nicefrac{n}{3}\rfloor} \left(x\cdot \S(x^3)\right)^{n-3i} + \sum_{i=\lfloor \nicefrac{n}{3}\rfloor+1}^{\lfloor \nicefrac{n}{2}\rfloor} \left(x^2\cdot \C(x^3)\right)^{3i-n}\right),
\end{split}
\]
from where we conclude the proof with arguments similar to those in the proof of Theorem~\ref{th:A026776}.
\end{proof}

\section*{Acknowledgements}
The work is supported by the National Science Foundation under grant No. IIS-1462107.
The work was awarded the John Riordan Prize (2015)~\cite{RiordanPrize} by the OEIS Foundation.

\bibliographystyle{acm} 
\bibliography{cat-sch.bib}

\end{document}